\documentclass{amsart}
\usepackage{amsrefs}

\usepackage{amsfonts}
\usepackage{amsmath}
\usepackage[arrow,matrix,tips, curve]{xy}
\usepackage{amsrefs}
\usepackage{amssymb}
\usepackage{amsthm}
\usepackage{latexsym}
\usepackage{url}
\usepackage{verbatim}

\theoremstyle{plain}
\newtheorem{theorem}{Theorem}
\newtheorem{proposition}[theorem]{Proposition}

\newtheorem{corollary}[theorem]{Corollary}

\theoremstyle{definition}
\newtheorem{definition}{\mdseries\scshape Definition}

\theoremstyle{remark}

\newtheorem{remark}{\mdseries\scshape Remark}

\begin{document}

\title[Igusa local zeta function for $x^n+y^m$]{The Igusa local zeta function for $x^n+y^m$}

\author{Rebecca Field}

\thanks{We would like to thank the NSF for the REU Site grant DMS-920084 that funded this project.}
\author{Vibhavaree Gargeya}

\author{Margaret M. Robinson} 
\author{Frederic Schoenberg}
\author{Ralph Scott}
\address{Department of Mathematics and Statistics, 
Mount Holyoke College,
50 College Street, South Hadley,
  MA 01075, USA}

\email{robinson@mtholyoke.edu}
\date{February 1994}

\begin{abstract}
 This paper provides specific results on the Igusa local zeta function for the curves 
$x^n+y^m$. In addition to specific results, we give an 
introduction to $p$-adic analysis and a discussion of various methods which have been 
used to compute these zeta functions.

\end{abstract}

\maketitle

\section{Introduction}

\pagestyle{plain}
\pagenumbering{arabic}
\setcounter{page}{1}

\parskip 5pt
\par
\noindent
This paper provides specific results on the Igusa local zeta 
function for the curves $f(x,y)=x^n+y^m$ where $m$ and $n$ are positive integers. We compute the following integral directly:
$$
Z(s)= \int_{{\bf Z}_p^2} |x^n+y^m|^s dx dy
$$
for $\mbox{Re} (s) >0$, where $| \cdot |$ represents the $p$-adic absolute value on the field of $p$-adic numbers, and $dx$ denotes the Haar measure on  normalized so that the measure of the $p$-adic integers ${\bf Z}_p$ is $1$. Igusa \cite{igusa-complex1}, \cite{igusa-complex2} showed that this local zeta function has a meromorphic continuation to the whole complex plane and is, in fact, a rational function of $t=p^{-s}$. We write $Z(s)=Z(t)$. 
In addition to these specific results the paper gives an 
introduction to the more general methods and techniques of $p$-adic analysis and 
introduces Igusa's $p$-adic stationary phase formula and the resolution of singularities which are two other methods that have been used to compute local zeta 
functions. For more general computations of local zeta functions for forms of various types, please see \cite{igusa-complexpow}, \cite{meuser}, \cite{lin}, \cite{hayes}, \cite{goldman}.
\par
This work was done in 1992 at the Summer Research Institute in Mathematics at Mount 
Holyoke College. Four undergraduates worked together under the direction of Margaret 
Robinson, who is on the faculty at Mount Holyoke College. The Institute was sponsored 
by the National Science Foundation and by the New England Consortium for 
Undergraduate Science Education. 
\par
After our summer's work, we were notified of the results of Professor C.Y. Lin
on an algorithm for computing the Igusa local zeta function for all curves using Igusa's 
$p$-adic stationary phase formula \cite{igusa-spf}. We hope that these specific examples will 
encourage interest in $p$-adic analysis and the questions surrounding the Igusa local 
zeta function.
\par
\noindent

\section{$P$-adic numbers.} In this paper, we will let $p$ stand 
for 
any fixed prime number, $p \in \{2,3,5,7,...\}$. We can construct the $p$-adic 
numbers, $ {\bf Q}_p$, from the rational numbers, ${\bf Q}$, in exactly the same as 
the way in which we construct the real numbers, $ {\bf R}$. 

To do real analysis, we construct the real numbers from the rational numbers by 
defining the real numbers as the limits of all Cauchy sequences of rational numbers. 
In this way we insure that all Cauchy sequences of rational numbers converge to real 
numbers, or, in other words, that the real numbers form a {\it complete} field. 
Hence, this construction is called completing the rational numbers with respect to 
the usual absolute value, $| \cdot |_\infty= \sqrt{x^2}$. When constructing ${\bf 
Q}_p$, we use a different absolute value, the $p$-adic absolute value. Thus, 
different sequences of rational numbers are Cauchy.

\noindent
\begin{definition}{\it The $p$-adic absolute value.}
A rational number, $a$, has a $p$-adic absolute value, $|a|_p$, such that:
\par
\noindent
$|a|_p = \left\{ \begin{array}{ll}
{1 \over p^{ord_p(a)}} & \mbox{if $a \not= 0$} \\
0    & \mbox{if $a =0.$}
\end{array}
\right. $
\par
\noindent
The quantity $ord_p a$ is called the order of $a$ and it is the highest power of $p$ 
dividing 
$a$. In other words, if $ a = p^n \cdot {\alpha \over \beta}$, where $\alpha$ and 
$\beta$ 
are integers that are not divisible by $p$, then $n= ord_pa$. 
\end{definition}
For example, if $p=5$ 
then $|75|_5=|5^2 \cdot 3|_5= 5^{-2}$, $|34|_5=|5^0 \cdot 34|_5=1$, and $|{1 \over 5}|_5=5$.

\noindent
\definition{\it A $p$-adic Cauchy sequence.}
A sequence of rational numbers, $\{ a_n\}$, is said to be a $p$-adic Cauchy sequence 
if given some $\epsilon >0$, there exists an integer $N$ such that for $i,j >N$, 
$|a_i-a_j|_p<\epsilon$.

Two $p$-adic Cauchy sequences, $\{a_n\}$ and $\{b_n\}$, are said to be equivalent if 
$|a_i-b_i|_p \to 0$ as $i \to \infty$. For example, if $p=5$ then 
$\{a_n\}=\{1,1,1,...\}$ and $\{b_n\}=\{1+5,1+5^2,1+5^3,...\}$ are equivalent since 
$|1-(1+5^i)|_5 \to 0$ as $i \to \infty$.

\noindent
\begin{definition} {\it The field of $p$-adic numbers, $ {\bf Q}_p$.}
The $p$-adic numbers, ${\bf Q}_p$, are defined to be all equivalence classes of $p$-
adic Cauchy sequences.
\end{definition}

Every equivalence class of Cauchy sequences in the real numbers has a standard 
representation, which we know as the decimal expansion. Likewise, it can be shown 
that every equivalence class of Cauchy sequences in the $p$-adic numbers has a unique 
representative of a certain form, which is called its $p$-adic expansion. The form of 
the decimal expansion does not determine a unique real number; for example, the 
sequences $\{1,1,...\}$ and $\{.9,.99,.999,...\}$ are equivalent Cauchy sequences 
which both represent the real number $1$. For a detailed explanation of the $p$-adic 
numbers and their expansions, see \cite{koblitz}.

\noindent
\begin{theorem} \cite{koblitz} {\it The ring of $p$-adic integers, ${\bf Z}_p$.}
Every equivalence class, $\alpha$, of sequences of rational integers has exactly one 
representative $p$-adic Cauchy sequence, $\{ \alpha_n \}$, where $\alpha_n \in {\bf Z}$, of 
the following form:
$$
\{\alpha_n\}=\{a_0, \ a_0+a_1p, \ a_0+a_1p+a_2p^2, \ ...\}
$$
where $\alpha_i =a_0+a_1p+a_2p^2+...+a_ip^i$ and $a_j \in \{0,1,...,p-1\}$, for $0 
\le j \le i$.
\end{theorem}

These equivalence classes with their unique representative form the $p$-adic 
integers, denoted ${\bf Z}_p$. It is convenient to think of such an equivalence class 
as its representative expansion and we denote $\alpha \in {\bf Z}_p$ as an infinite series:
$\alpha = a_0+a_1p+a_2p^2+...$. This series representation of the $p$-adic integers 
is analogous to the decimal expansion for real numbers, and we often write $\alpha 
=a_0.a_1 a_2a_3...$. We define the norm of $\alpha \in {\bf Z}_p$, to be $0$ if 
$\alpha $ represents $0$ and in the non-zero case $|\alpha|_p=p^{-ord_p(\alpha)}$, 
where $ord_p(\alpha)$ is the exponent of the lowest power of $p$ with non-zero 
coefficient in the $p$-adic expansion of $\alpha$. Any rational integer, $n$, is a 
$p$-adic integer because it is the limit of the constant sequence $\{n,n,...\}$ and 
has a terminating $p$-adic expansion (to signify the repeating sequence), for 
example, for $p=5$, $405=1 \cdot 5+1 \cdot 5^2+3 \cdot 5^3=5(1+1 \cdot 5+3 \cdot 
5^2)$ and 
$|405|_5=5^{-1}$. Thus, the $ord_p(\alpha)$ is the highest power of $p$ which divides 
the $p$-adic expansion of $\alpha$. However, there are more $p$-adic integers then 
the rational integers, for example, if $p=5$ then $1/6$ is also an integer because it 
can be written as $1+4 \cdot 5+4 \cdot 5^2+...$.(To check that this expansion 
represents 
$1/6$ multiply it by $6=1+1\cdot5$ to get $1$.) However, not every rational 
number is a $p$-adic integer. The number $36/5=1/5(1+2 \cdot5 +1 \cdot 5^2)$ is not a $5$-adic integer; it is an 
element in ${\bf Q}_5$ of absolute value $5$. It can be shown that ${\bf Q}_p$, as we defined it above, is equivalently defined as 
the field of fractions of ${\bf Z}_p$. Hence, we can think of elements in ${\bf Q}_p$ 
as:
$$
{\alpha \over \beta}= {a_0+a_1p+a_2p^2+...\over b_0+b_1p+b_2p^2+...}
                    =p^{-(ord(\alpha)-ord(\beta))}u
$$
where $\alpha \in {\bf Z}_p$, $\beta \in {\bf Z}_p-\{0\}$ and $u$ is a unit in ${\bf 
Z}_p$. The units in ${\bf Z}_p$ are the invertible elements and they have $p$-adic 
absolute value $1$. In other words, $u$ is a unit if the first digit in its $p$-adic 
expansion is not $0$. For example, we can find the $5$-adic expansion of the number 
$30/225$ by long division (where we think of $0$ as $5+4 \cdot 5+4 \cdot 5^2+...$ and 
borrow from the right if we need to) as follows:
$$
{30 \over 
225}={5(1+1\cdot5)\over5^2(4+1\cdot5)}={1\over5}(4+1\cdot5+3\cdot5^2+1\cdot5^3+3\cdot
5^4+...)
$$
In the representation above, $6/9=2/3$ is a $5$-adic unit and we can sum the 
convergent geometric series above to show that we have the correct expansion.
$$
4 +1\cdot5+3\cdot5^2+1\cdot5^3+3\cdot5^4+...=4+5(1+3\cdot5)\sum_{i=0}^\infty
5^{2i}=4+5(16)\lim_{n\to \infty} {1-5^{2(n+1)} \over 1-5^2}={2\over3}
$$
For more examples of numerical computations with the $p$-adics see \cite{bachman}, \cite{koblitz},\cite{mahler}.

The $p$-adic integers, ${\bf Z}_p$, form a ring with base $p$ addition and 
multiplication, where carrying is done to the right. The set of $p$-adic integers with first 
digit equal to $0$ are divisible by $p$, hence denoted $p{\bf Z}_p$, and form the 
maximal ideal in ${\bf Z}_p$. The ideal $p^2{\bf Z}_p$ is the set of $p$-adic 
integers with first two digits equal to $0$ and $p^e{\bf Z}_p$ is the set of $p$-adic 
integers with first $e$ digits equal to $0$. Hence, we see that ${\bf Z}_p \supset 
p{\bf Z}_p \supset p^2{\bf Z}_p \supset ...\supset p^e{\bf Z}_p \supset ...\supset 
\{0\}$ is a descending chain of ideals. The units in ${\bf Z}_p$ are denoted by ${\bf Z}_p-p{\bf Z}_p$ since they are 
elements in ${\bf Z}_p$ but not in $p{\bf Z}_p$. The units can be written as a 
disjoint union of the $p-1$ cosets of $p{\bf Z}_p$: ${\bf Z}_p-p{\bf Z}_p=\coprod_{a\in \{1,2,..p-1\}} 
(a+p{\bf Z}_p)$. The $p$-adic integers themselves are an infinite disjoint union: 
$$
{\bf Z}_p = \coprod_{e=0}^\infty p^e{\bf Z}_p -p^{e+1}{\bf Z}_p
          = \coprod_{e=0}^\infty p^e({\bf Z}_p-p{\bf Z}_p)
$$

The $p$-adic numbers can be visualized as an infinite system of concentric circles 
representing the ideals $p^e{\bf Z}_p$ for all integer values of $e$ and corresponding 
to the inclusions $\{0\}\subset ...\subset p^2{\bf Z}_p \subset p{\bf Z}_p
\subset {\bf Z}_p \subset p^{-1}{\bf Z}_p \subset ...$. Since $0$ is contained in all 
the ideals, it lies at the center of all the circles. If the order of a $p$-adic number 
$\alpha$ is $e$ then $\alpha$ falls into the $p^e{\bf Z}_p$ circle but not into the 
$p^{e+1}{\bf Z}_p$ circle and hence we think of it as an element in the annulus 
$p^e({\bf Z}_p-p{\bf Z}_p)$. In this way, we see that the disjoint union above for the 
$p$-adic integers is really a union of annuli, and, more generally, ${\bf Q}_p$, itself, 
is a union of infinitely many disjoint annuli, each annulus containing the elements in the 
field of the same absolute value:
$$
{\bf Q}_p = \coprod_{e=-\infty}^\infty p^e{\bf Z}_p -p^{e+1}{\bf Z}_p
          = \coprod_{e=-\infty}^\infty p^e({\bf Z}_p-p{\bf Z}_p).
$$

\par
\noindent

\section{The definition of the Igusa local zeta function and $p$-adic analysis.}
The Igusa local zeta function, $Z(s)$,  associated to a polynomial in the ring of 
polynomials with $n$ variables and coefficients in ${\bf Z}$, $f(x) = 
f(x_1,x_2,...,x_n) \in {\bf Z}[x_1,x_2,..,x_n]$, is defined for a complex variable 
$s$, where $Re(s)>0$, as :
$$
Z(s)= \int_{{\bf Z}_p^n} |f(x)|_p^s dx.
$$
As $|f(x)|_p^s$ is always a power of $p^{-s}$, we can think of $Z(s)$ as a function 
of $t=p^{-s}$ and write $Z(t)$.

Before we continue, we must explain what we mean by integration in this setting.

\noindent
\begin{definition} {\it Measure of the set $E$.} Let $E$ be a union of sets 
$U_1,U_2,...U_m$ of the form $U_i= a_i +p^{N_i}{\bf Z}_p$ (analogous 
to intervals in real analysis), where $a_i \in {\bf Z}_p$, $N_i$ is any non-negative 
integer, and $1 \le i \le m$. We define the measure of the set $E$, $m(E)$, to be the 
integral over $E$ (written $m(E)=\int_E dx$) such that the following four 
properties hold:

1. $m(E)= \int_E dx \ge 0$ and $m(\emptyset)=0$

2. Translation invariance. If $a \in {\bf Z}_p$, $m(a +E)=\int_{a+E} dx = \int_E dx 
=m(E)$

3. If $U_i \cap U_j = \emptyset$, then $m(U_i \cup U_j)=\int_{U_i \cup U_j} dx= 
\int_{U_i} dx + \int_{U_j} dx = m(U_i)+m(U_j)$

4. $m({\bf Z}_p)=\int_{{\bf Z}_p} dx =1$.
\end{definition}

This measure is called the Haar measure on ${\bf Q}_p$ normalized so that $m({\bf Z}_p)=1$ and it is known to be the 
unique measure to satisfy these four properties.
\par
\noindent
\begin{proposition} {\it The measure of the ideal $p^e {\bf Z}_p$ is $p^{-e}$ for 
any integer $e$.} 
\end{proposition}
\begin{proof}
We break up ${\bf Z}_p$ into its $p^e$ cosets modulo $p^e {\bf Z}_p$ and use 
properties 2, 3, and 4 of the measure. Each coset is of the form $b_0+b_1 p +b_2 p^2 
+...+b^{e-1}p^{e-1} +p^e {\bf Z}_p$ and, as there are $p$ choices for each $b_i$, we 
get $p^e$ cosets. Hence,
$$
1= \int_{{\bf Z}_p} dx = \sum_{b \in {\bf Z}_p-p^e {\bf Z}_p} \int_{b+p^e{\bf Z}_p} 
dx =p^e \int_{p^e {\bf Z}_p} dx.
$$
From this argument, we see that $m(p^e {\bf Z}_p) = p^{-e} $
\end{proof}
\par
\noindent
\begin{remark} This proposition gives a change of variables formula which we write 
as $d(p^ex) = |p^e|_p dx = p^{-e}dx$. In fact, by the same proof the following,  more 
general proposition is true: If $E$ is a measurable set of the type given above then 
the measure of the set $p^e E$ is $p^{-e} m(E)$.
\end{remark}

\section{The Igusa local zeta function for the polynomial $f(x)=x^N$.} 
\par
Break ${\bf Z}_p$ into a disjoint union (as above) of annuli where $|x^N|_p^s$ is 
constant. Thus,
$$
Z(t)= \int_{{\bf Z}_p} |x^N|_p^s dx = \sum_{e=0}^\infty \int_{p^e({\bf Z}_p-p{\bf 
Z}_p)}|x^N|_p^s dx.
$$
Now, making the change of variables $x = p^e y$ for $y \in {\bf Z}_p-p{\bf Z}_p$
and $dx=p^{-e}dy$. We see that
$$
Z(t)= \sum_{e=0}^\infty p^{-e}\int_{{\bf Z}_p-p{\bf Z}_p} |(p^e y)^N|_p^s dy
=\sum_{e=0}^\infty p^{(-1-Ns)e} \int_{{\bf Z}_p-p{\bf Z}_p}dy
= {1 \over 1-p^{-1-Ns}} \int_{{\bf Z}_p-p{\bf Z}_p} dy
$$
As the units are a disjoint union of their $p-1$ cosets modulo $p{\bf Z}_p$, we have 
that
$$
\int_{{\bf Z}_p-p{\bf Z}_p} dy = \sum_{c=1}^{p-1} \int_{c+p{\bf Z}_p} dy =(p-
1)\int_{p{\bf Z}_p}dy = 1-p^{-1}.
$$
And we have that the Igusa local zeta function for $f(x)=x^N$ is precisely
$$
Z(t)= {1-p^{-1} \over 1-p^{-1-Ns}} = {1-p^{-1} \over 1-p^{-1} t^N}.
$$
This integral is extremely well-known and appears in Tate's Thesis \cite{tate}.
\par
\noindent
\section{ Local zeta functions and their Poincare series.}
As defined above, the Igusa local zeta function associated to the polynomial $f(x)$ 
in $n$ variables with coefficients in the integers ( or equivalently in the $p$-adic 
integers) has a Poincare series which contains arithmetic information about the 
cardinality of solutions of $f(x)$ in the ideals $p^e {\bf Z}_p$.

Before we consider the Poincare series, let's look more closely at $Z(t)$.
The zeta function is really an infinite sum of the measures of all points in ${\bf 
Z}_p^n$ that are mapped into each annulus of ${\bf Z}_p$ by the polynomial 
$f(x)$.Hence,
$$
Z(t)=\int_{{\bf Z}_p^n } |f(x)|_p^s dx = \sum_{e=0}^\infty m(\{x\in {\bf Z}_p^n | 
f(x)=p^e u \}) \ t^e
$$
where $u \in {\bf Z}_p-p{\bf Z}_p$.
From this expression for $Z(t)$, it is clear that $Z(1)=1$ and that $Z(0)= 
m(\{x|f(x)=u\})$. If we know this first measure and $f(x)$ is a homogeneous polynomial whose partial derivatives vanish modulo $p$ only at $0$, then by Hensel's Lemma we can compute the 
measures of the other sets and find the rational function for $Z(t)$ in terms of this 
first number. The problems in computing the Igusa local zeta function come when the 
polynomial has complicated singularities.

The Poincare series which is associated to the Igusa local zeta function is a 
generating function for the number of solutions of $f(x)$ modulo $p^e$ for all non-
negative integer $e$. It is defined as:
$$
P(t)=\sum_{e=0}^\infty |N_e| \ p^{-ne}t^e
$$
where $n$ is the number of variables in $f(x)$, $ |N_e|$ is the cardinality of the 
set: $N_e =\{a \in {\bf Z}_p^n-p^e{\bf Z}_p^n | f(a) \equiv 0 \bmod p^e\}$, and 
$N_0 =1$.
Borewicz and Shafarevic conjecture the rationality of $P(t)$ in \cite{borewicz}, p. 47. Igusa proves this conjecture in \cite{igusa-complex1}, \cite{igusa-complex2} by 
showing that $Z(t)$, and hence $P(t)$, is rational.

\noindent
\begin{proposition} \cite{igusa-complex1},\cite{igusa-complex2} The local zeta function associated to a 
polynomial is related to the polynomial's Poincare series as:
$$
Z(t)=P(t)- {1 \over t}(P(t)-1).
$$
\end{proposition}
\noindent
\begin{proof} If $f^{-1}(p^e {\bf Z}_p)$ is the pull-back (or inverse image) of 
$p^e{\bf Z}_p$ then
$$
f^{-1}(p^e{\bf Z}_p)=\{x \in {\bf Z}_p^n | f(x) \equiv 0 \bmod p^e \}=
\coprod_{a \in N_e} (a +p^e {\bf Z}_p).
$$
Hence, $m(f^{-1}(p^e {\bf Z}_p))-m(f^{-1}(p^{e+1}{\bf Z}_p)) $ is the measure of the 
set of points in ${\bf Z}_p^n$ that map to the annulus $p^e{\bf Z}_p-p^{e+1}{\bf 
Z}_p$. With this in mind, we see the following:
\begin{align*}
Z(t) &= \int_{{\bf Z}_p^n } |f(x)|_p^s dx \\
&= \sum_{e=0}^\infty \int_{f^{-1}(p^e {\bf Z}_p)-f^{-1}(p^{e+1}{\bf Z}_p)} p^{-es} 
dx \\
&= \sum_{e=0}^\infty p^{-es}[m(f^{-1}(p^e {\bf Z}_p))-m(f^{-1}(p^{e+1}{\bf Z}_p))].
\end{align*}
We remember that $f^{-1}(p^e {\bf Z}_p))= \coprod_{a \in N_e} (a +p^e{\bf Z}_p)$. The 
number $|N_e|$ is the number of centers $ a \in {\bf Z}_p^n - p^e{\bf Z}_P^n$ such 
that $f(a) \equiv 0 \bmod p^e$. The measure of a disk $a+p^e{\bf Z}_p^n$ is $p^{-en}$ 
since $\int_{a+{\bf Z}_p^n} dx =\int_{a_1+{\bf Z}_p} dx_1 \int_{a_2+{\bf Z}_p} dx_2 
...\int_{a_n+{\bf Z}_p} dx_n =p^{-en}$. Thus, $m(f^{-1}(p^e{\bf Z}_p))=|N_e|p^{-en}$. 
Now we see that
\begin{align*}
Z(t) &= \sum_{e=0}^\infty |N_e| (p^{-n}t)^e -\sum_{e=0}^\infty |N_{e+1}|p^{-
(e+1)n}t^e \\
&= P(t)-t^{-1}(P(t)-1).
\end{align*}
\end{proof}
\section{ Direct computation of the Igusa local zeta for $x^n+y^m.$}
As an example to familiarize the reader with $p$-adic analysis, we will calculate  the 
Igusa local zeta function for $f(x,y) = x^n + y^m$ directly. We will show that
$$
Z(t) = \frac{(1-p^{-1}){p^{-1}}}{(1-p^{-{m+n\over (m,n)}}t^{{mn\over 
(m,n)}})}  [ \sum_{j=0}^{{n\over (m,n)}-1}p^{-\lfloor jm/n \rfloor - j}t^{mj} + 
\sum_{j=0}^{{m\over (m,n)}-1}p^{-\lfloor jn/m \rfloor-j} t^{nj} ]
$$
$$ 
 +{[(p-1)^2-(|N(0)|-1)]p^{-2}(1-p^{-1}t) + (|N(0)|-1)(p^2-p)p^{-4}t \over (1-p^{-1}t) (1-p^{-{m+n \over (m,n)}} t^{mn \over (m,n)})},
$$
First we break ${\bf Z}_p$ up into its infinitely many annuli $\coprod_{e= 0}^{\infty} p^e 
({\bf Z}_p - p{\bf Z}_p)$ and write it as

$$
 Z(t) = \int_{{\bf Z}_p} \sum_{e= 0}^{\infty} \int_{p^e({\bf Z}_p - p{\bf Z}_p)}|x^n 
 + y^m|_p^s dxdy.
$$
Similarly, we break up the region of integration for the other variable, and make the 
change of variables $x \rightarrow p^ex$ and $y\rightarrow p^fy$, to get that

$$
Z(t) = \sum_{f= 0}^{\infty} \sum_{e= 0}^{\infty} p^{-e} p^{-f} \int_{({\bf Z}_p - 
p{\bf Z}_p)^2} |p^{ne}x^n + p^{mf}y^m|_p^s dxdy.
$$
Note that $x$ and $y$ are now both units in $({\bf Z}_p - p{\bf Z}_p)$.
To compute the integral above, we break it into three partial integrals. We consider 
the three cases:
1) $ne > mf $,
2) $mf > ne $, and
3) $mf = ne $.

\par
In the first case where $ne > mf $, we have a factor of $p^{mf}$ in the integrand. In 
this case, we can write the partial integral as 
$$
I_1(t)=\sum_{f\ge 0} \sum_{e> mf/n} p^{-e} p^{-f} \int_{({\bf Z}_p - p{\bf Z}_p)^2} 
|p^{mf} (p^{ne - mf}x^{n} + y^m)|_p^s dxdy 
$$

Since $y$ is a unit, the integral is equal to the measure of $({\bf Z}_p - p{\bf 
Z}_p)^2$ which is $(1-p^{-1})^2$ and we have that
$$
I_1(t) = \sum_{f\ge 0} \sum_{e > mf/n} p^{-e}p^{-f}t^{mf}(1-p^{-1})^2
$$
Now,
\begin{eqnarray}\label{sum}
\sum_{e > mf/n} p^{-e} = \sum_{e=0}^\infty p^{-e} - \sum_{e=0}^{\lfloor mf/n \rfloor} 
p^{-e}
\end{eqnarray}

We need to sum from $e=0$ to the greatest integer in $mf/n$ which we will denote by 
$\lfloor mf/n \rfloor$. This sum depends upon the congruence class of $f$ modulo $n 
\over (m,n)$ where $(m,n)$ is the greatest common factor of $m$ and $n$. If $f={nk 
\over (m,n)}$ where $k$ is a positive integer then we sum from $e=0$ to $mk \over 
(m,n)$. If $ f = {nk \over (m,n)}+ 1$ then $mf/n = {mk \over (m,n)} + m/n$
and we sum from $e=0$ to ${mk \over (m,n)} + \lfloor m/n \rfloor$. 
If $f = {nk \over (m,n)} + 2 $ then $ mf/n = {mk \over (m,n) }+ 2m/n$
and we sum from $e=0 $ to $ {mk \over (m,n)}+ \lfloor 2m/n \rfloor$.
If $f={nk \over (m,n)}+ j$ where $0 \le j \le {n\over (m,n)}-1$ then $mf/n =
{ mk 
\over (m,n)} + jm/n$
and we sum from $e=0$ to $ {mk \over (m,n)} + \lfloor jm/n \rfloor$.

Therefore, 
$$
\sum_{e=0}^{\lfloor mf/n \rfloor}p^{-e} = \sum_{j=0}^{{n \over (m,n)}-1} \
\sum_{e=0}^{{mk \over (m,n)}+ \lfloor jm/n \rfloor} p^{-e} = \sum_{j=0}^{{n \over 
(m,n)}-1} {(1-p^{-{mk\over (m,n)}-\lfloor jm/n \rfloor -1}) \over (1-p^{-1})} 
$$

Using formula $(\ref{sum})$ above, we have that the inner sum in $I_1(t)$ is
$$
\sum_{e >mf/n} p^{-e} = \sum_{j=0}^{{n\over (m,n)}-1} {p^{-{mk\over (m,n)}-\lfloor 
jm/n \rfloor-1} \over (1-p^{-1})}
$$

Using this sum, $I_1(t)$ becomes

\begin{eqnarray*}
 I_1(t) & = &(1-p^{-1})\sum_{k\ge 0}\sum_{j=0}^{{n\over(m,n)}-1}p^{-{mk\over (m,n)}-
 \lfloor jm/n \rfloor-1}p^{-{nk\over (m,n)}-j}t^{{mnk\over (m,n)}+mj}\\
 & = &\frac{(1-p^{-1})p^{-1}}{(1 - p^{-{m+n\over (m,n)}}t^{{mn\over 
 (m,n)}})}\sum_{j=0}^{{n\over (m,n)}-1}p^{-\lfloor jm/n \rfloor - j}t^{mj}
\end{eqnarray*}

Similarly in the second case, when $mf > ne$, the second partial integral becomes:
$$
I_2(t)=\frac{(1-p^{-1})p^{-1}}{1 - p^{-{m+n\over (m,n)}}t^{{mn\over 
(m,n)}}}\sum_{j=0}^{{m\over (m,n)}-1}p^{-\lfloor jn/m \rfloor - j }t^{nj}
$$ 

Therefore, we see that 
$$
I_1(t)+I_2(t) = \frac{(1-p^{-1}){p^{-1}}}{(1-p^{-{m+n\over (m,n)}}t^{{mn\over 
(m,n)}})} \  [ \sum_{j=0}^{{n\over (m,n)}-1}p^{-\lfloor jm/n \rfloor - j}t^{mj} + 
\sum_{j=0}^{{m\over (m,n)}-1}p^{-\lfloor jn/m \rfloor-j} t^{nj} ]
$$

Finally in the third and most difficult case where $mf=ne$,

\begin{eqnarray*}
I_3(t) & = & \sum_{f \ge 0, \ {mf \over n} \in {\bf Z}} p^{-mf/n-f} t^{mf} 
\int_{({\bf 
Z}_p - p{\bf Z}_p)^2} |f(x,y)|^s dxdy \\
\ & = & \sum_{k=0}^\infty 
(p^{-{m+n \over (m,n)}} t^{mn \over (m,n)} )^k  \int_{({\bf Z}_p - p{\bf Z}_p)^2} 
|f(x,y)|^s dxdy \\
\ & = & {1 \over (1-p^{-{m+n \over (m,n)}} t^{mn \over (m,n)})} \int_{({\bf Z}_p - 
p{\bf Z}_p)^2} |f(x,y)|^s dxdy
\end{eqnarray*}
and we will show that
$$
I_3^\prime (t)=\int _{({\bf Z}_p - p{\bf Z}_p)^2} |f(x,y)|^s dxdy 
= [(p-1)^2-(|N(0)|-1)]p^{-2} + {(|N(0)|-1)(1-p^{-1})p^{-2}t \over (1-p^{-1}t)},
$$
where $f(x,y) = x^m+y^n$, and $p$ does not divide both $m$ and $n$, $|N(0)|$ is 
the cardinality of the set $N(0)=\{(x,y) \in {\bf F}_p \times {\bf F}_p \ | \ x^m+y^n 
\equiv 0 \bmod p \}=N_1+1$, and ${\bf F}_p$ is the finite field with $p$ elements.
\par
\noindent
{\bf First proof using counting argument}
\par
\noindent
${({\bf Z}_p - p{\bf Z}_p)^2}= \coprod_{k=0}^\infty \{ (x,y) \in  U^2 | 
|f(x,y)|^s = t^k \} $, where $t = p^{-s}$.
So, 
$$
I_3^\prime (t) = \sum_{k=0}^\infty t^k \ m \{ (x,y) \in U^2 \ | \ |f(x,y)|^s = t^k 
\}.
$$ 
Hence, 
\begin{eqnarray*}
I_3^\prime(t)
= \lim_{r \to \infty} \sum_{k=0}^r t^k \ m \{(x,y) & \in & (c_o+c_1p+...+c_{r-1}
p^{r-1})+p^r{\bf Z}_p \\
& \ & \times  (d_0+d_1p+...+d_{r-1}p^{r-1})+p^r{\bf Z}_p \\
& \ &  | \ c_0,d_0 \in  
{\bf F}_p^*, c_i,d_i \in  {\bf F}_p \ for \ i>0, |f(x,y)|^s = t^k \}.
\end{eqnarray*}
This interpretation means that 
\begin{eqnarray*}
I_3^\prime (t)&=& \lim_{r\to\infty} \sum _{k=0}^r
\ card \{(x,y) \in  (c_0+c_1p+...+c_{r-1}p^{r-1}) \times (d_0+d_1p+...+d_{r-1}p^{r-
1}) \\
& |& \
|f(x,y)|^s = t^k \} \cdot \ t^k\  m(p^r{\bf Z}_p)^2.
\end{eqnarray*}
Now since $|f(x,y)|^s = t^k$ if and only if $ord_p(x^m+y^n)=k$. If we let 
$|ord=k|_r$ denote 
the cardinality of the set $\{ (x,y) \in  (c_0+c_1p+...+c_{r-1}p^{r-
1})\times (d_0+d_1p+...d_{r-1}p^{r-1}) \ | \ |f(x,y)^s = t^k \}.$
Then we see that $I_3^\prime(t)= \lim_{r\to\infty} \sum _{k=0}^r |ord=k|_r t^k 
p^{-2r}$.

What is $|ord=k|_r$?
\begin{eqnarray*}
|ord=0|_r = \ card \{(x,y) & \in &  (c_0+c_1p+...+c_{r-1}p^{r-1}) \\
&\times& (d_0+d_1p+...+d_{r-1}p^{r-1}) \ | \ ord_p(x^m+y^n) = 0 \}.
\end{eqnarray*}
Now  $ord_p(x^m+y^n) = 0$ if and only if $p$ does not divide $(c_0+c_1p+...c_{r-
1}p^{r-1})^m + 
(d_0+d_1p+...+d_{r-1}p^{r-1})^n$ which happens if and only if $p$ does not divide 
${c_0}^m + {d_0}^n$ by the binomial theorem.  Thus, 
$c_1,...,c_{r-1}$ and $d_1,...,d_{r-1}$ do not influence the condition 
$ord_p(x^m+y^n) = 0$ and are each free to take any value in ${\bf F}_p$. With this in mind, 
we see that 
$$
|ord=0|_r= card \{(c_0,d_0) \in ({\bf F}_p^{*})^2 \ | \  
{c_0}^m+{d_0}^n)\equiv 0 \bmod p \} \ p^{2(r-1)}
= [ (p-1)^2 - (N(0)-1) ] p^{2r-2}.
$$
$$
|ord=1|_r = \{(x,y) \in  (c_0+c_1p+...+c_{r-1}p^{r-1})\times(d_0+d_1p+...+d_{r-
1}p^{r-1}) \ | \ ord_p(x^m+y^n) = 1 \}
$$
Now, $ord_p(x^m+y^n) = 1$ if and only if $p$ divides $ x^m+y^n$  and $p^2 $ does not 
divide $x^m+y^n$ which happens if and only if
$p$ does divide $c_0^m+d_0^n$ but $p^2$ does not divide $(c_0+c_1p+...+c_{r-1}p^{r-
1})^m+(d_0+d_1p+...+d_{r-1}p^{r-1})^n$. This last condition means that 
$p^2$ does not divide $(c_0+c_1p)^m + 
(d_0+d_1p)^n$ which by the binomial theorem implies that $p^2$ does not divide 
$ {c_0}^m + mc_0c_1p 
+ {d_0}^n + nd_0d_1p$.
Now, if $ p$ divides $ {{c_0}^m+{d_0}^n}$, then 
${c_0}^m+{d_0}^n = p^e$, for some $e \ge 1$ and we have that $ord_p(x^m+y^n)=1$ if.
$ p^2$ 
does not divide $ p^e + mc_0c_1p + nd_0d_1p$.
This last condition implies that $p$ does not divide $p^{e-1} + mc_0c_1 + nd_0d_1$.  
Now, $c_0$ and $d_0$ are units, so if $m$ is a unit, then $p^2 / {x^m+y^n}$ implies 
that $c_1 \equiv (-p^{e-1} - nd_0d_1)m^{-1}c_0^{-1} \bmod p$.  Thus, $d_1$ can be any 
element of 
${\bf F}_p$, and $c_1$ is determined by $d_1$. So, there are $p$ choices for 
$(c_1,d_1) \in  {\bf F}_p^2$ such that $p^2$ divides $x^m+y^n$, and therefore $(p^2-
p)$ choices such that $p^2 $ does not divide $x^m+y^n$.  Similarly, if $n$ is a unit, 
then there 
are again $(p^2-p)$ possibilities for $(c_1,d_1)$ such that $p^2$
does not divide $x^m+y^n$.  
Finally, 
\begin{eqnarray*}
|ord=1|_r &=& card \{ (c_0,d_0)\in ({\bf F}_p^*) \ | \  {x^m+y^n} \equiv 0 \bmod p\} \\ 
&\times& card \{( c_1,d_1)\in ({\bf F}_p^2) \ | \ x^m+y^n \not\equiv 0 \bmod p^2 \} \\
&\times& \ card \{(c_2,...,c_{r-1},d_2,...,d_{r-1}) \in ({\bf F}_p)^{2(r-2)} \} \\
&=& (N(0)-1) (p^2-p)  p^{2r-4}.
\end{eqnarray*}
In general, for $k \ge 1$, $|ord=k|_r = card\{(c_0,d_0)\in ({\bf F}_p^{*})^2 \ |
{x^m+y^n} \equiv 0 \bmod p\} \cdot \ card \{(c_1,d_1)\in {\bf F}_p^2 \ | \ x^m+y^n
\equiv 0 \bmod p^2 \}  \cdots \ card \{(c_{k-
1},d_{k-1}) \ |  x^m+y^n \equiv 0 \bmod p^{k}\} \cdot card\{(c_k,d_k) \ | x^m+y^n 
\not\equiv 0 \bmod p^{k+1}\} 
\cdot p^{2(r-1-k)}.$
And hence
$$
|ord=k|_r= (N(0)-1) (p)^{k-1} (p^2-p) p^{2(r-k-1)}.
$$
Now, we can compute our integral.
\begin{eqnarray*}
I_3^\prime(t) &=& \lim_{r\to\infty} p^{-2r} \sum_{k=0}^r 
|ord=k|_r t^k \\
&=& \lim_{r\to\infty} p^{-2r} [ ((p-1)^2-(N(0)-1))p^{2r-2} + \sum 
_{k=1}^r ((N(0)-1)p^{k-1}(p^2-p)p^{2(r-k-1)}) t^k ] \\
&=& ((p-1)^2-(N(0)-1))p^{-2} + \lim_{r\to\infty} p^{-2r} \sum _{k=0}^r ((N(0)-
1)p^k(p^2-
p)p^{2(r-k-2)}t^{k+1}) \\
&=& ((p-1)^2 - (N(0)-1))p^{-2} + 
(N(0)-1)(p^2-p)p^{-4}t \sum _{k=0}^\infty p^{-k}t^k \\
&=& ((p-1)^2 - (N(0)-1))p^{-2} + {(N(0)-1)(1-p^{-1})p^{-2}t \over 1-p^{-1}t}.
\end{eqnarray*}
In the case where $m=2$ and $n=3$,
$N(0) = p$, so
$$
I_3^\prime(t) = ((p-1)^2 - (p-1))p^{-2} + {(p-1)(1-p^{-1})p^{-2}t 
\over 1-p^{-1}t}
= \frac{(p-1)(p-2)p^{-2} + t(1-p^{-1})p^{-2}}{1-p^{-1}t}.
$$
\par
\noindent
{\bf Second proof using a change of variables}
\par
\noindent
We will compute $I_3^\prime (t)=\int_{{({\bf Z}_p - p{\bf Z}_p)^2}}|x^m+y^n|^s \ dxdy$ using another method.
Write the integral as a sum over its cosets modulo $p$ and count those points for which $x^m+y^n$ is and is not $0$ modulo $p$.
\begin{eqnarray*}
\int_{{({\bf Z}_p - p{\bf Z}_p)^2}}|x^m+y^n|^s \ dxdy&=&
\sum_{(a,b)\in {\bf F}_p^*\times {\bf F}_p^*}\int_{(a,b)+p{\bf Z}_p^2}|x^3+y^2|^s \ dxdy\\
&=& \left((p-1)^2-(|N(0)|-1)\right)\int_{(p{\bf Z}_p)^2} \ dxdy \\
& \ & \ + \ (|N(0)|-1)p^{-2} \int_{{\bf Z}_p^2}|(a+px)^m+(b+py)^n|^s \ dxdy
\end{eqnarray*}
Expand $(a+px)^m+(b+py)^n$ to get $|(a^m+ma^{m-1}px+ \ldots +p^mx^m)+(b^n+nb^{n-1}py+ \ldots +p^ny^n)|=|p(c + ma^{m-1}x+nb^{n-1}y+p( \ldots)|$
where $a^m+b^n=cp$ for some $c \in {\bf F}_p^*$. There is a measure-preserving transformation $(x,y) \to (x^*,y^*)$ such that $x^*=g_1(x,y)=ma^{m-1}x+nb^{n-1}y+p( \ldots)$ and $y^*=g_2(x,y)=y$. The Jacobian of this transformation is 
\begin{eqnarray*}
\det\left|
\begin{array}{ll}
\frac{\partial g_1}{\partial x}& {\frac{\partial g_1}{\partial y}} \\
\frac{\partial g_2}{\partial x}& {\frac{\partial g_2}{\partial y}}
\end{array}
\right|&=&\det\left| \begin{array}{ll}
\frac{\partial g_1}{\partial x}(a,b)& {\frac{\partial g_1}{\partial y}}\\
0& {1}
\end{array}
\right|\\
&=&\frac{\partial g_1}{\partial x}(a,b)=ma^{m-1}
\end{eqnarray*}
If $ma^{m-1}={0 \bmod p}$, then let $g_1(x,y)=x$ and $g_2=ma^{m-1}x+nb^{n-1}y+p( \ldots)$ so that  the Jacobian becomes $nb^{n-1} \not = 0 \bmod p$ since $p$ does not divide both $m$ and $n$. For a general description of this measure-preserving map see [8,9].
Now, we see that $\left|f(a+px,b+py)\right|^s \ = \ \left|p\left(x^*+c\right)\right|^s$ where $dxdy=dx^*dy^*$. Returning to the integral above,
\begin{eqnarray*}
I_3^\prime (t)&=&((p-1)^2-(|N(0)|-1))p^{-2}+ (|N(0)|-1)p^{-2}t \int_{{\bf Z}_p^2}\left|x^*+c\right|^s \ dx^*dy^*\\
&=&((p-1)^2-(|N(0)|-1))p^{-2}+ \ {(|N(0)|-1)p^{-2}t (1-p^{-1}) \over 1-p^{-1}t}
\end{eqnarray*} 
\par
\noindent
Hence, $Z(t)= I_1(t)+I_2(t)+I_3(t)$ and we have computed the local zeta function in terms of $|N(0)|$. Below we compute $|N(0)|$ for our curve.
\par
\noindent
\begin{theorem} For any prime $p$ the number of solutions to $x^m+y^n=0$ in ${\bf 
F}_p$, which we have called $|N(0)|$ above, is
$$
|N(0)| = \left\{ \begin{array}{ll}
1+(p-1) gcd(m,n,p-1) & \mbox{if $p=2$ or }\\
\quad & \mbox{if $ord_2(p-1)>min \{ord_2(m),ord_2(n) \}$ } 
\\
\\
1    & \mbox{otherwise.}
\end{array}
\right. 
$$
\end{theorem}
\par
\noindent
\begin{proof} The case $p=2$ is trivial. 

If $ord_2(p-1)>min \{ord_2(m),ord_2(n) \}$ 
then w.l.o.g. assume $ord_2(p-1)>ord_2(m)$. The set $B_t = \{x^t | x \in {\bf F}_p^* 
\}$ of $t$ powers forms a subgroup of order $(p-1)/gcd(p-1,t)$ in ${\bf F}_p^*$.
Let $Z=B_m \bigcap B_n$ then $Z$ is a subgroup of order $gcd(|B_m|,|B_n|)={(p-
1)gcd(p-1,m,n) \over gcd(p-1,n)gcd(p-1,m)}$. Since $ord_2(p-1)>ord_2(m)=Y$, 
$2^{Y+1}$ divides $p-1$ and there is an element $x$ in ${\bf F}_p^*$ which 
generates the cyclic subgroup of order $2^{Y+1}$ in ${\bf F}_p^*$. Since $x$ 
generates this subgroup, $x^{2^Y} \equiv -1$ and as $m/2^{Y}$ is odd we see that 
$x^m=x^{2^Y {m \over 2^Y}} \equiv (-1)^{m \over 2^Y} = -1$.

Allowing for the point $(0,0)$, we have that $|N(0)|=1+|\{(a,b) \in {\bf F}_p^* 
\times {\bf F}_p^* | a^m+b^n \equiv 0 \bmod p \}|$. By the existence of the 
particular element $x$ above, we have that $|N(0)|=1+|\{(a,b) \in {\bf F}_p^* \times 
{\bf F}_p^* | a^m \equiv b^n \bmod p \}|$ since for each $a$ where $a^m=c$ there is a 
unique $ax$ where $(ax)^m=-c$. As $a^m \equiv b^n$ implies that $a^m$ and $b^n$ are 
in 
the subgroup $Z$ above, we can rewrite $|N(0)|=1+\sum_{z \in Z}|\{ (a,b) \in {\bf 
F}_p^* \times {\bf F}_p^* | a^m \equiv b^n \equiv z \bmod p \}|=1+\sum_{z \in Z} \{ a 
\in {\bf F}_p^* | a^m \equiv z \}| \times |\{b \in {\bf F}_p^* | b^n \equiv z \}|$. 
For each $z$ in $Z$, $|\{ a \in {\bf F}_p^* | a^m \equiv z \}| = gcd(p-1,m)$ and $|\{ 
b \in {\bf F}_p^* | b^n \equiv z \}| = gcd(p-1,n)$. Therefore, $|N(0)|=1+|Z| \times 
gcd(p-1,m)gcd(p-1,n)=1+(p-1)gcd(p-1,m,n)$.

If $ord_2(p-1) \le min \{ ord_2(m), ord_2(n) \}$ and $c=2^{ord_2(p-1)}$, we know 
that $c$ divides $m$ and $n$. Now if we have $(a,b)\not= (0,0)$ such that $a^m+b^n 
\equiv 0 \bmod p$ then $0 \equiv a^m+b^n=a^{cu}+b^{cv}$ implies that $(a^u)^c \equiv 
-(b^v)^c$. As $b \not=0$, we have the element $y=a^u(b^{-1})^v$ such that 
$y^c\equiv -1$. Thus $y^{2c} \equiv 1$ and the order of $y$ must divide $2c = 
2^{ord_2(p-1)+1}$. If the order of $y < 2c$ then the order of $y$ would divide $c$ 
implying the impossible $y^c \equiv 1$. Therefore, the order of $y$ must be equal to 
$2c$ but as $2c$ does not divide $p-1$, we have a contradiction and the only solution 
to $a^m+b^n \equiv 0 \bmod p $ is $(0,0)$. Hence, in this case $|N(0)|=1.$
\end{proof}

\section{ Computation of the Igusa local zeta function for $f(x,y)=x^n+y^m$ using the $p$-adic stationary phase formula (SPF)}
\par
\noindent
\begin{theorem} {\it A $p$-adic stationary phase formula}  \cite{igusa-spf}. The local zeta function of a polynomial $f(x)$ in $n$ variables with coefficients in ${\bf Z}_p$ is:
$$
Z(t)= (p^n-|\bar N(0)|)p^{-n} +(|\bar N(0)|-|\bar S|)p^{-n}t {1-p^{-1} \over 1-p^{-1}t} 
+ \int_{x \in S} |f(x)|^s dx
$$
where $| \bar N(0)|$ is the cardinality of the set of $x$ in ${\bf F}_p^n$ for which $f(x) = 0 \bmod p$, $|\bar S|$ is the cardinality of the set of singular $x$ in $\bar N(0)$ such that all partial derivatives of $f$ are $0$ modulo $p$ at these points, and $S$ is the set of all $x$ in ${\bf Z}_p^n$ which are congruent to vectors in $\bar S \bmod p$.
\end{theorem}
\par
\noindent
\begin{corollary} {\it To compute $Z(t)= \int_{{\bf Z}_p\times {\bf Z}_p}|x^n+y^m|^s \ dxdy$, SPF must be applied $\frac{m}{(m,n)}+\frac{n}{(m,n)}-1$ times.}
\end{corollary}
\par
\noindent  
\begin{proof} For the curve $x^m +y^n$, SPF will terminate when the original integral appears as the singular integral. We may assume that $m<n$. For if $m=n$, SPF will terminate after one application.  The first application of SPF results in
$$\int_{{\bf Z}_p^2}|x^n+y^m|^sdxdy \ = \ I_1 \ + \ I_2 \ + p^{-2}\int_{{\bf Z}_p^2}|p^nx^n+p^mx^m|^s \ dxdy$$
Where $I_1$ and $I_2$ are known in terms of $|N(0)|$.
Taking out the factor of $p^m$ the singular integral becomes $p^{-2}t^m \ \int_{{\bf Z}_p\times {\bf Z}_p}|p^{n-m}x^n+y^m|^s \ dxdy$.
Applying  SPF to this integral produces the known terms and a new singular integral $p^{-1}\int_{{\bf Z}_p^2}|p^{n-m}x^n+p^my^m|^s \ dxdy$. To proceed, we must determine whether $n-m>m$ or $n-m<m$.  In other words, we have to know which is bigger $n$ or $2m$.  
\par
The string of SPF applications ends when instead of $>$ or $ <$ we get $=$ in which case, all powers of $p$ factor out of the integrand and the singular integral becomes $\int_{{\bf Z}_p\times {\bf Z}_p}|x^n+y^m|^s \ dxdy$.  Taking a multiple of $Z(t)$ to the other side, we are now able to solve for the rational function.
\par
With each application of SPF we get a new inequality of the form $kn<lm$ or $kn>lm$. Suppose we may assume this to be the inequality for the SPF of $$\int_{{\bf Z}_p\times {\bf Z}_p}|p^{kn-(l-1)m}x^n+p^my^m|^s \ dxdy,$$ so since $kn-(l-1)m<m$, the integral reduces to 
$$t^{kn-(l-1)m}\int_{{\bf Z}_p\times {\bf Z}_p}|x^n+p^{m-(kn-(l-1)m)}y^m|^s \ dxdy.$$
And our next application of SPF produces the integral 
$$\int_{{\bf Z}_p\times {\bf Z}_p}|p^nx^n+p^{lm-kn}y^m|^s \ dxdy$$
which can be simplified after determining which is greater, $(k+1)n$ or $lm$.  As we can see, each application of SPF increases either $k$ or $l$ by $1$, and since SPF must be applied exactly once more when $k=\frac{m}{(m,n)}$ and $l=\frac{n}{(m,n)}$ and $k=l=1$ at the start, the total number of applications of SPF will be $\frac{m}{(m,n)}+\frac{n}{(m,n)}-1.$        
\end{proof}                                 			       

\section{Computation of the Igusa local zeta for $x^n+y^m$ using resolution of singularities.}
Resolve $f((x,y))=x^3+y^2$ to find the zeta function 
$$Z(t)=\int_{{\bf Z}_p\times {\bf Z}_p}|f((x,y))|^s \ dxdy \ = \ \sum_{D_1D_n}\int_{h^{-1}({\bf Z}_p\times {\bf Z}_p)}|f\circ h|^s|h^*(dx\wedge dy)|$$
Where $h^*$ is the pull back of $h$ and $h^{-1}({\bf Z}_p\times {\bf Z}_p)=D_i$ for some i.
Let $h:M\rightarrow {\bf Z}_p^2$ such that
$$h((x_1,y_1))=(x,\frac{y}{x})$$ and $$h((\xi_1,\eta_1))=(y,\frac{x}{y}).$$
Therefore, $$h^{-1}((x,y))=\left\{
\begin{array}{ll}
(x_1,x_1y_1)& \mbox{condition described below}\\ 
(\xi_1\eta_1,\xi_1)& \mbox{condition described below}
\end{array}
\right. $$
So, $$f((x,y))=\left\{\begin{array}{l}
x_1^2(x_1+y_1^2)\\
\xi_1^2(1+\xi_1\eta_1^3)
\end{array}
\right.$$

Resolve $f((x,y))=x_1^2(x_1+y_1^2).$
$$h((x_2,y_2))=(y_1,\frac{x_1}{y_1})$$ 
$$h((\xi_2,\eta_2))=(x_1,\frac{y_1}{x_1})$$
Therefore, $$h^{-1}((x_1,y_1)=\left\{
\begin{array}{l}
(x_2y_2,x_2)\\
(\xi_2,\xi_2\eta_2)
\end{array}
\right.$$
So,$$f((x,y))=\left\{
\begin{array}{l}
x_2^3y_2^2(x_2+y_2)\\
\xi_2^3(1+\xi_2\eta_2^2)
\end{array}
\right.$$

Now, we resolve $f((x,y))=x_2^3y_2^2(x_2+y_2)$
$$h((x_3,y_3))=(x_2,\frac{y_2}{x_2})$$  $$h((\xi_3,\eta_3))=(y_2,\frac{x_2}{y_2})$$
Therefore, $$h^{-1}((x_2,y_2))=\left\{
\begin{array}{l}
x_3^6y_3^2(1+y_3)\\
\xi_3^6\eta_3^3(1+\eta_3)
\end{array}
\right.$$
Neither of these curves need to be resolved further.
Now we must find the $D_i$.
Since x and y are in ${\bf Z}_p$, $|x|_p \le 1, \ |y|_p \le 1$, so either $|\frac{x}{y}|_p\le 1$ or $|\frac{y}{x}|_p<1$.  Therefore $|x_1|\le 1, |y_1|\le 1$ and $|\xi_1|\le 1, |\eta_1|<1$.  This means that $(x_1,y_1) \in {\bf Z}_p^2$ and $(\xi_1,\eta_1) \in {\bf Z}_p\times p{\bf Z}_p$.  Therefore, $D_1={\bf Z}_p\times p{\bf Z}_p$ and the same argument holds for $D_2, D_3$, while $D_4={\bf Z}_p^2$.

\noindent
Next we calculate the change in measure due to the change of variables in $D_1$.
\begin{align*}
dx\wedge dy&=d(\xi_1,\eta_1)\wedge d\xi_1\\
\quad &=\xi_1 \ d\eta_1\wedge d\xi_1\\
\quad &=-\xi_1 \ d\xi_1\wedge d\eta_1
\intertext{Now, calculate the first piece of the zeta function over $D_1$.}
Z_1(t)&=\int_{{\bf Z}_p\times p{\bf Z}_p}|\xi_1^2(1+\xi_1\eta_1^3)|^s \ |\xi_1| \ d\xi_1d\eta_1\\
\quad &=\int_{{\bf Z}_p\times p{\bf Z}_p}|\xi_1|^{2s+1} \ d\xi_1d\eta_1\\
\quad &=p^{-1} \ \int_{{\bf Z}_p}|\xi_1|^{2s+1} \ d\xi_1\\
\quad &=p^{-1} \frac{1-p^{-1}}{1-p^{-2}t^2}
\intertext{Next we calculate the change in measure due to the change of variables in $D_2$.}
dx \wedge dy &=dx_1\wedge d(x_1y_1)\\
&=x_1 \ dx_1\wedge dy_1\\
&=\xi_2 \ d\xi_2\wedge d(\xi_2\eta_2)\\
&=\xi_2^2 \ d\xi_2\wedge d\eta_2 
\intertext{Now, calculate  the second piece of the zeta function over $D_2$.}
Z_2(t)&=\int_{{\bf Z}_p\times p{\bf Z}_p}|\xi_2^3(1+\xi_2\eta_2^2)|^s \ |\xi_2|^2 \ d\xi_2d\eta_2\\
&=\int_{{\bf Z}_p\times p{\bf Z}_p}|\xi_2|^{3s+2} \ d\xi_2d\eta_2\\
&=p^{-1} \frac{1-p^{-1}}{1-p{-3}t^3}\\
\intertext{Next we calculate the change in measure due to the change of variables in $D_3$.}
dx \wedge dy&= x_2y_2 \ d(x_2y_2)\wedge dx_2\\
&= -x_2^2y_2 \ dx_2\wedge dy_2\\
&=-(\xi_3\eta_3)^2\xi_3 \ d(\xi_3\eta_3)\wedge d\xi_3 \\
&=\xi_3^4\eta_3^2 \ d\xi_3\wedge d\eta_3 
\intertext{Now, calculate  the third piece of the zeta function over $D_3$.}
Z_3(t)&=\int_{{\bf Z}_p\times p{\bf Z}_p}|\xi_3^6\eta_3^3(1+\eta_3)|^s \ |\xi_3|^4|\eta_3|^2 \ d\xi_3d\eta_3\\
&=\int_{{\bf Z}_p}|\xi_3|^{6s+4} \ d\xi_3 \ \  \int_{p{\bf Z}_p}|\eta_3|^{3s+2} \ d\eta_3\\
&=p^{-3}t^3 \left(\frac{1-p^{-1}}{1-p^{-5}t^6}\right)
\left(\frac{1-p^{-1}}{1-p^{-3}t^3}\right)\\
\intertext{Next we calculate the change in measure due to the change of variables in $D_4$.}
dx\wedge dy&=-x_3^2x_3y_3 \ dx_3\wedge d(x_3y_3)\\
&= -x_3^4y_3 \ dx_3\wedge dy_3
\intertext{Now, calculate  the fourth piece of the zeta function over $D_4$.}
Z_4(t)&=\int_{{\bf Z}_p^2}|x_3^6y_3^2(1+y_3)|^s \ |x_3|^4|y_3| \ dx_3dy_3\\
&=\int_{{\bf Z}_p}|x_3|^{6s+4} \ dx_3 \ \ \int_{{\bf Z}_p}|y_3|^{2s+1}|y_3+1|^s \ dy_3\\
&=\frac{1-p^{-1}}{1-p^{-5}t^6} \ \sum_{a mod p} \ \int_{a+p{\bf Z}_p}|y_3|^{2s+1}|y_3+1|^s \ dy_3\\
&=\left(\frac{1-p^{-1}}{1-p^{-5}t^6}\right)\left((p-2)p^{-1} \ + \ \int_{p{\bf Z}_p}|y_3|^{2s+1} \ dy_3 \ + \ \int_{-1+p{\bf Z}_p}|y_3+1|^s \ dy_3\right)\\
&=\frac{1-p^{-1}}{1-p^{-5}t^6}
\left((p-2)p^{-1} \ + \ p^{-2}t^2\frac{1-p^{-1}}{1-p^{-2}t^2} \ + \ p^{-1}t\frac{1-p^{-1}}{1-p^{-1}t}\right)
\end{align*}

The Igusa local zeta function, $Z(t)$, is the sum of $Z_1,Z_2,Z_3,Z_4$, or 
$$Z(t)=\frac{(1-p^{-1})(1-p^{-2}t+p^{-2}t^2-p^{-5}t^6)}
{(1-p^{-1}t)(1-p^{-5}t^6)}$$

\end{document}